\documentclass[11pt, reqno]{amsart}
\usepackage[english]{babel}
\usepackage[table]{xcolor}
\usepackage[utf8]{inputenc}
\usepackage{amsmath, amsfonts, amsthm, amscd, amssymb, fullpage, upgreek, enumerate, comment, siunitx, enumitem, graphicx, mathrsfs, mathtools}
\usepackage{changepage}
\usepackage{relsize}
\usepackage{caption}
\makeatletter
\newenvironment{caseindent}[1][2em]
  {%
    \begin{list}{}{%
      \setlength{\leftmargin}{#1}%
      \setlength{\rightmargin}{0pt}%
      \setlength{\topsep}{0pt}%
      \setlength{\partopsep}{0pt}%
      \setlength{\itemsep}{0pt}%
      \setlength{\parsep}{\parskip}%
      \setlength{\listparindent}{\parindent}%
    }%
    \item[]%
    
    \ignorespaces
  }
  {%
    \end{list}%
    \ignorespacesafterend
  }
\makeatother

\usepackage[parfill]{parskip}
\usepackage{hyperref}

\newcommand{\leanblueprintbase}{https://mj141592.github.io/weighted-chains-paper}
\newif\ifleanlinks
\leanlinkstrue
\newcommand{\leanblueprint}[1]{%
  \ifleanlinks
    \leavevmode\unskip\penalty9999\hbox{}\nobreak\hfill\quad
    \href{\leanblueprintbase/theorems/#1/}{%
      \mbox{\normalfont\footnotesize\textcolor{blue}{[\textsf{Lean formalisation}]}}}%
    \par\nobreak
  \fi
}

\newtheorem{theorem}{Theorem}[section]
\newtheorem{proposition}[theorem]{Proposition}
\newtheorem{lemma}[theorem]{Lemma}
\newtheorem{definition}[theorem]{Definition}
\newtheorem*{claim}{Claim}

\theoremstyle{remark}
\newtheorem*{remark}{Remark}

\numberwithin{equation}{section}

\newcommand{\bburl}[1]{\textcolor{blue}{\url{#1}}}

\newcommand{\hr}[1]{\href{#1}{\url{#1}}}

\newcommand{\cube}{\{0, 1, 2\}^n}
\newcommand{\thehypercube}{\{0, 1, \ldots, d\}^n}
\newcommand\modd{\operatorname{mod}}

\newcommand{\fclass}{\mathchoice
  {\left\lfloor \frac{nd}{2} \right\rfloor}
  {\big\lfloor \frac{nd}{2} \big\rfloor}
  {\big\lfloor \frac{nd}{2} \big\rfloor}
  {\big\lfloor \frac{nd}{2} \big\rfloor}}
\newcommand{\cclass}{\mathchoice
  {\left\lceil \frac{nd}{2} \right\rceil}
  {\big\lceil \frac{nd}{2} \big\rceil}
  {\big\lceil \frac{nd}{2} \big\rceil}
  {\big\lceil \frac{nd}{2} \big\rceil}}

\newcommand{\thespecialpoint}{(\frac{d}{2}, \ldots, \frac{d}{2})}
\newcommand{\theonepoint}{(1, \ldots, 1)}
\newcommand{\wn}{W_{n}}
\newcommand{\wno}{W_{n-1}}
\newcommand{\un}{U_{n}}
\newcommand{\uno}{U_{n-1}}

\newcommand{\nocontentsline}[3]{}

\title{A generalisation of Sperner's theorem using weighted chain decomposition}

\author{Ya\"el Dillies}
\thanks{Department of Mathematics, Stockholm University. Email: yael.dillies@math.su.se}

\author{Matthew Johnson}
\thanks{Independent researcher, London, UK. Email: mjohnson31415926@gmail.com}

\author{Aleksandra Kowalska}
\thanks{Mathematical Institute, University of Oxford. Email: aleksandra.kowalska@seh.ox.ac.uk}

\begin{document}

\begin{abstract}
    We discuss a generalisation of Sperner's theorem in the following form: for any $1 < k \leqslant n$ and $d \leqslant 2$ we find the unique largest subset $A$ of $\thehypercube$ such that for any $x, y \in A$, if $x \leqslant y$ coordinatewise, then $x$ and $y$ differ on more than $k$ coordinates.

    To do so, we introduce a decomposition of $\thehypercube$ into weighted chains, which (among other desirable properties) is invariant under coordinate permutation.

    A special case of our result gives a new proof of Sperner's theorem.

    The results of this paper have been formalised in Lean.
\end{abstract}

\maketitle

\section{Introduction}

Sperner's theorem \cite{sperner} says that the largest family $A$ of subsets of $[n]$, without two subsets where one contains the other, is of size $\binom{n}{\lfloor n/2 \rfloor}$, and the families
\begin{equation*}
    \big\{X \subseteq [n]: \vert X\vert = \left\lfloor \frac{n}{2} \right\rfloor\big\} \ \text{and} \ \big\{X \subseteq [n]: \vert X\vert = \left\lceil \frac{n}{2} \right\rceil\big\}
\end{equation*}
are the unique such families (where for $2 \mid n$ they are equal).

There are a couple of intuitive ways to generalise this result. The first one is to relax the condition that no subset may contain another one. For example, we may allow $A$ to have two sets $X, Y \in A$ with $X \subseteq Y$ only if $\vert Y \setminus X \vert > k$ for some constant $k$ (or only if $\vert Y \setminus X \vert \leqslant k$). Katona studied this and related problems in \cite{katona}, where he classified all largest families of subsets satisfying such conditions.

Another possible way of generalising Sperner's result is to work on multisets instead of sets. In this case, we wish to find the largest family of multisubsets of a multiset $\{1^{d_1}, \ldots, n^{d_n}\}$ (i.e., a multiset containing $d_i$ copies of $i$ for $1 \leqslant i \leqslant n$) such that for no two $X, Y \in A$ we may have $X \subsetneq Y$ (i.e., $Y$ containing at least as many copies of each number as $X$). This problem (under an equivalent formulation) was studied and solved by de Bruijn, Tengbergen and Kruyswijk \cite{symmetric_chain_decomposition}.

In this paper, we consider a generalisation of Sperner's theorem which pursues both these directions at once. Before we state it, let us note that Sperner's theorem may be reformulated in a more geometric language as follows: what is the largest subset of $\{0, 1\}^n$ which has no elements $x, y$ with $x \leqslant y$ coordinatewise? (Let us call such a subset of $\{0, 1\}^n$ or $\thehypercube$ an \emph{antichain}.)

We may formulate a generalisation of Sperner's theorem which explores both of the mentioned notions at once as follows:

Let us call a subset $A \subseteq \thehypercube$ \emph{$k$-separated} if for any different $x, y \in A$ with $x \leqslant y$, $x$ and $y$ differ on strictly more than $k$ coordinates. What is the largest $k$-separated subset of the hypercube $\thehypercube$?

We now note in passing that the notion of $k$-separation is based on the $L_0$ metric (the number of coordinates which differ) rather than the $L_1$ metric (the sum of coordinate differences). The $L_1$ case follows from Katona's proof in \cite{katona} (and in fact from many other techniques used to prove Sperner's theorem, for example a generalisation of Lubell's technique \cite{lubell} as presented in \cite{nagys}). On the other hand, the $L_0$ formulation resisted attempts at generalisations of existing proofs and seemed to require new ideas, and hence it drew our interest.

For $k \geqslant n$, a family is $k$-separated if and only if it is an antichain. Hence, in what follows we will assume that $k \leqslant n$ (as any larger $k$ gives the same condition as $k=n$). 

For $k=1$ the problem is equivalent to finding the largest subset of the hypercube with no two elements in the same `row' (i.e., placing non-attacking rooks on an $n$-dimensional chessboard). Such a subset must have at most $(1+d)^{n-1}$ elements (as otherwise, by the pigeonhole principle, there would be two with the same first $n-1$ coordinates). On the other hand, there are many subsets of this size satisfying this condition, for example
\begin{equation*}
    \big\{ x \in \thehypercube: L \cdot x \equiv 0 \ (\modd \ (d+1)) \big\}
\end{equation*}
for any vector $L$ with coefficients coprime to $d+1$.

In this paper, we prove the following (and conjecture that it holds for any $d \geqslant 1$):

\begin{theorem}
\label{wc:thm:main-theorem}
\label{the_main_theorem}
For any $n$, $1 < k \leqslant n,$ and $d \in \{1, 2\}$, the unique largest $k$-separated subsets of the hypercube are:
\begin{equation}
    A_1 := \big\{x \in \thehypercube: \vert x\vert  \equiv \fclass \ (\modd \ dk+1) \big\}
\end{equation}
and
\begin{equation}
    A_2 := \big\{x \in \thehypercube: \vert x\vert  \equiv \cclass \ (\modd \ dk+1) \big\},
\end{equation}
where we denote
\begin{equation*}
    \vert x\vert  := \sum_i x_i.
\end{equation*}

If $2 \mid nd$, we have $A_1=A_2$ and so there is only one such largest subset.
\leanblueprint{main-theorem}
\end{theorem}

\begin{remark}
    For any $x, y \in A$ with $x \neq y$ and $x_i < y_i$ on at most $k$ coordinates, we have
    \begin{equation*}
        0 < \vert y\vert -\vert x\vert  = \sum_i (y_i-x_i) \leqslant dk
    \end{equation*}
    (as $x$ and $y$ differ on at most $k$ coordinates, and their difference on each is at most $d$). Hence, we must have $\vert x\vert  \not\equiv \vert y\vert  \ (\modd dk+1)$.

    Hence, the sets $A_1, A_2$ as defined in Theorem \ref{the_main_theorem} are indeed $k$-separated.

    Since $x \in A_1$ if and only if $(d-x_1, \ldots, d-x_n) \in A_2$, so clearly $\vert A_1\vert =\vert A_2\vert $.
\end{remark}

The case $d=1$ (in a different, but equivalent formulation) easily follows from the results of Katona \cite{katona} and Griggs \cite{griggs}. However, we present our (novel and arrived at independently) proof of this case, as we hope that it is of independent interest and as it generalises very naturally to the $d=2$ case.

The case $d=1, k=n$ is simply Sperner's theorem. Our proof simplifies in this special case, and we present it in Appendix \ref{appendix_sperner}, which may be read independently from the rest of the paper.

Finally, we note that for any $d$ and $\frac{n}{2} \leqslant k \leqslant n$ there exists a simple proof that the sets $A_1, A_2$ in Theorem \ref{the_main_theorem} are largest possible, presented in Appendix \ref{appendix_large_k} (as it is rather unrelated to the rest of the paper and doesn't show their uniqueness). This case also follows from \cite[Theorem 5.1]{tour_m_part_p_sperner_families}.

\subsection{AI usage note}
In working on and writing this paper, we used LLMs for spelling and typo checking, grammar proofreading, \LaTeX{} queries, searching literature, searching for binomial coefficient identities, generating \LaTeX{} code for the figures (corresponding to our detailed descriptions) and generating the Lean formalisation code and the corresponding blueprint (with GPT-5.6 Sol in ultra mode). We have checked manually that the Lean code is a faithful translation of the paper, and that the proofs are legitimate and follow the arguments of the paper. The formalisation process had the additional benefit of uncovering several further minor typos in the statements and proofs. As a further experiment, we attempted a few efforts with both GPT-5.6 Sol and GPT-6 Astra in ultra mode at proving the cases beyond our main theorem, i.e. with $d \geqslant 3$, and the models made no noteworthy progress. No part of the mathematical content or the paper write-up was generated by a model.

\subsection{The structure of the paper}
In the rest of this section we discuss related research. Section \ref{section_strategy} introduces the definitions used throughout the paper. Then, in Section \ref{section_assigning_preliminary}, we discuss our main proof strategy: decomposing the hypercube into weighted chains. Proposition \ref{the_weights_assigning_proposition} from that section implies Theorem \ref{the_main_theorem} provided an appropriate assignment of weights to chains can be found, which is done in the rest of the paper.

In Sections \ref{section_d1} and \ref{section_d2}, we show assignments of weights satisfying the conditions of the Proposition in the cases $d=1$ and $d=2$ respectively and prove their desired properties.

Finally, Section \ref{section_conclusion} contains conclusions and a discussion of generalising our results to cases with $d>2$.

Appendix \ref{appendix_sperner} contains a new proof of Sperner's theorem which is a special case of the proof from this paper.

\subsection{Another generalisation of Sperner's Theorem}

We briefly discuss another related problem, which was our initial motivation for the notion of $k$-separated sets.

Let $X_1, \ldots, X_n$ be disjoint sets. What is the size of the largest family of subsets $\mathcal{F} \subseteq \mathcal{P}(X_1 \cup \ldots \cup X_n)$ such that there are no $U, V \in \mathcal{F}$ with $U \subsetneq V$ and $U \cap X_i = V \cap X_i$ for some $1 \leqslant i \leqslant n$?

Our method of solving the problem above (using symmetric chain decomposition similar to \cite{symmetric_chain_decomposition}) required the case $k=n$ of the problem presented in this paper. Having solved the special case with $k=n$ (with a simpler method for large $k$ described in Appendix \ref{appendix_large_k}), we tried to generalise it to other values of $k$, which motivated the work whose results are presented in this paper.

The above problem has been discussed and solved in a paper \cite[Theorem 5.1]{tour_m_part_p_sperner_families} on various similar generalisations of Sperner's theorem, and hence we do not pursue this direction here. Its special case with $n=2$ was first solved by Katona \cite{katona_n2}.

\subsection{Connections to coding theory}

The problem we are studying can be reformulated as follows in the language of coding theory: suppose that we have an alphabet $\{0, 1, \dots, d\}$ and an asymmetric channel, where the error of swapping $x$ for $y$ can occur only if $x<y$. What is the maximum size of a code with codewords of length $n$ that would detect any erroneous transmission with at most $k$ errors?

Error detection in asymmetric channels was studied by Ahlswede and Aydinian in \cite{asymmetric_channels}. They studied the problem above for binary codes over the alphabet $\{0, 1\}$, and arrived at the case $d=1$ of Theorem \ref{the_main_theorem}.

For larger alphabets they generalised this problem in a different direction: they studied asymmetric channels where all numbers from the alphabet $\{0, 1, \dots, d\}$ are transmitted as their binary expansions and the only errors that can occur are transmitting bit $1$ instead of bit $0$, so $x$ can be swapped for $y$ if and only if on all positions in the binary expansions where they differ $y$ has a $1$ and $x$ has a $0$. We note that this is a stronger condition than $x < y$. In \cite[Proposition 6]{asymmetric_channels} they observe that the set $A_1$ (as defined in Theorem \ref{the_main_theorem}) is error-detecting for codes of length $n$ where at most $k$ errors can occur, and note that larger such sets exist.

Gunby, He, Narayanan and Spiro in \cite{more_coding_theory} studied, from the perspective of coding theory, problems related to Sperner's theorem and reminiscent of its generalisation explored here. They looked for the largest antichain $A$ in $\{0, 1\}^n$ satisfying a further condition that any $x, y \in A$ have Hamming distance at least $k$ (i.e., differ on at least $k$ indices) and show that $\vert A\vert=O_k(2^nn^{-k/2})$ for odd $k$.

Finally, in \cite{on_error_detection_in_asymmetric_channels}, Kovačević introduces a general coding-theory framework, where one is given an asymmetric channel over the alphabet $\{1, \dots, q\}$ and wishes to find the (asymptotically) largest error-detecting code, where the errors can appear on at most $h$ coordinates, be of magnitude at most $a$ each and have the sum of magnitudes at most $t$. This phrasing is a generalisation of the problem discussed in this paper.

\subsection*{Acknowledgements}
We would like to thank our supervisors Ben Green and Tom Bloom for helpful discussions and support.

\section{Preliminaries}
\label{section_strategy}

In this section we discuss some definitions and ways of looking at the hypercube that are later useful for assigning weights satisfying the conditions of Proposition \ref{the_weights_assigning_proposition}. The discussion in this section holds for any $d \geqslant 1$. 

We will call the elements of $\thehypercube$ \emph{points} and \emph{vertices} (of the hypercube) interchangeably.

\begin{definition}[Chain]
\label{wc:def:chain}
    We call an ordered set of elements $x_1, x_2, \ldots, x_l \in \thehypercube$ a chain if $x_1 \leqslant x_2 \leqslant \ldots \leqslant x_l$.

    We call $l$ (the number of the chain's elements) its length.
\leanblueprint{chain}
\end{definition}

\begin{definition}[Saturated chain]
\label{wc:def:saturated-chain}
\label{saturated_definition}
    Let us call a chain $x_1 \leqslant x_2 \leqslant \ldots \leqslant x_l$ saturated if for any $i$ we have $\vert x_{i+1}\vert -\vert x_i\vert =1$, i.e. all consecutive elements differ on exactly one coordinate by $1$.
\leanblueprint{saturated-chain}
\end{definition}

\begin{definition}[The width of a chain]
\label{wc:def:chain-width}
    We say that a chain $x_1 \leqslant x_2 \leqslant \ldots \leqslant x_l$ is of width $w$ if there are exactly $w$ indices $i$ such that $x_{1,i} < x_{l, i}$.
\leanblueprint{chain-width}
\end{definition}

If in a chain $x_1 \leqslant x_2 \leqslant \ldots \leqslant x_l$ the first and last elements have equal coordinate $i$ ($x_{1,i}=x_{l,i}$), then all the elements of the chain must have the same coordinate $i$ ($x_{1, i} = x_{2, i} = \ldots = x_{l, i}$).

Hence, the width of a chain is the number of coordinates which ever change in the chain (i.e. take more than one value among the elements of the chain).

For $d=1$, a saturated chain's width is $l$ if and only if its length is $l+1$.

Let us also note that no chain of width at most $k$ may contain two elements of a $k$-separated set.

\begin{definition}[Symmetric chain]
\label{wc:def:symmetric-chain}
    We say that a chain $x_1 \leqslant x_2 \leqslant \ldots \leqslant x_l$ is symmetric if for any $i$ we have $\vert x_{i}\vert +\vert x_{l+1-i}\vert =nd$ (hence, if the corresponding elements from the beginning and end of the chain are equally far from the middle of the cube).
\leanblueprint{symmetric-chain}
\end{definition}

\begin{remark}
    A saturated chain is symmetric if and only if the condition $\vert x_1\vert +\vert x_l\vert =nd$ holds.
\end{remark}

It is often convenient to think of the points in the cube as grouped into diagonal \emph{layers}:

\begin{definition}[A layer]
\label{wc:def:layer}
    For $0 \leqslant i \leqslant nd$ let the $i$th layer of the hypercube be defined as:
    \begin{equation*}
        L_i := \{x \in \thehypercube: \vert x\vert =i\}.
    \end{equation*}
\leanblueprint{layer}
\end{definition}

We note that the sizes of layers are symmetric ($\vert L_i \vert = \vert L_{nd-i} \vert$), as $(x_1, \dots, x_n) \in L_i$ if and only if $(d-x_1, \dots, d-x_n) \in L_{nd-i}$. They are also unimodal:
\begin{equation*}
    1=\vert L_0\vert  \leqslant \vert L_1\vert  \leqslant \ldots \leqslant \vert L_{\fclass}\vert =\vert L_{\cclass}\vert  \geqslant \ldots \geqslant \vert L_{nd-1}\vert  \geqslant \vert L_{nd}\vert =1,
\end{equation*}
which can be shown by induction on $d$ with base case $d=1$ (when $\vert L_i\vert=\binom{n}{i}$ and so the unimodality follows from the unimodality of binomial coefficients).

Let us call the layer $L_{\frac{nd}{2}}$ (if $2 \mid nd$) or the layers $L_{\fclass}$ and $L_{\cclass}$ (if $2 \nmid nd$) the \emph{middle layers} or \emph{the middle of the cube}.

\begin{definition}[Good chains]
\label{wc:def:good-chain}
\label{good_chain_definition}
    For given $n, k, d$ let us call a chain $C=(x_1, \ldots, x_l) \subseteq \thehypercube$ \emph{good} if it is saturated, of width at most $k$ and either symmetric or of length exactly $dk+1$.
\leanblueprint{good-chain}
\end{definition}

The shortest good chains are (for $2 \mid nd$) the singletons
\begin{equation}
\label{singleton_chains_equation}
    \{x\} \ \text{for} \ \vert x\vert =\frac{nd}{2}
\end{equation}
or (for $2 \nmid nd$) pairs of vertices
\begin{equation}
\label{double_chains_equation}
    \{x, y\} \ \text{for} \ x \leqslant y, \ \vert x\vert =\fclass=\vert y\vert -1.
\end{equation}

Any good chain contains exactly one element from $A_1$ and exactly one from $A_2$.

\begin{definition}[A chain starting at $x$]
\label{wc:def:chain-start}
    We say that a chain $x_1 \leqslant x_2 \leqslant \ldots \leqslant x_l$ starts at $x_1$ if $\big\vert  \vert x_1\vert -\frac{nd}{2} \big\vert  \geqslant \big\vert  \vert x_l\vert -\frac{nd}{2} \big\vert $.

    If $\big\vert  \vert x_1\vert -\frac{nd}{2} \big\vert  \leqslant \big\vert  \vert x_l\vert -\frac{nd}{2} \big\vert $, we say that the chain starts at $x_l$.

    If an equality holds in the above two inequalities, we may say either that the chain starts at $x_1$ or at $x_l$.

    That is, a chain starts at $x$ if it passes only through points at least as far from the middle of the cube as $x$.
\leanblueprint{chain-start}
\end{definition}

\begin{definition}[The lower and upper sides of the hypercube]
\label{wc:def:cube-sides}
    We will often refer to $\{x \in \thehypercube: \vert x\vert  \leqslant \frac{nd}{2}\}$ as the lower side of the cube, and to $\{x \in \thehypercube: \vert x\vert  \geqslant \frac{nd}{2}\}$ as the upper side of the cube.
\leanblueprint{cube-sides}
\end{definition}

We note that points $x$ with $|x| = \frac{nd}{2}$ according to this convention belong both to the lower and upper sides of the cube.

The following notion will turn out to be crucial for assigning weights to chains.
\begin{definition}[Inner and outer layers]
\label{wc:def:inner-outer-layers}
    Let us call a layer $L_i$ inner if the length of symmetric saturated chains between layers $L_i$ and $L_{nd-i}$ is smaller than $kd+1$. Otherwise, we call a layer outer.
\leanblueprint{inner-outer-layers}
\end{definition}

\begin{remark}
    A layer is inner if and only if $(nd-i)-(i-1) < kd+1$, but we will not use this formula very often.
    
    Moreover, a layer is outer if and only if there are no saturated chains of width at most $k$ (hence, of length at most $kd+1$) starting on the other side of the cube and crossing it.
\end{remark}

For $d \leqslant 2$, the outer layers are those for which the good chains starting at them have the full length $kd+1$, and the inner layers are those for which the good chains starting at them are the shorter, symmetric chains.

We note that when $k=n$, the only outer layers are $L_0=\{(0, \ldots, 0)\}$ and $L_{dn}=\{(d, \ldots, d)\}.$

\begin{definition}[Type]
\label{wc:def:type}
    For $x \in \thehypercube$, we say that $x$ is of type $(t_0, \ldots, t_d)$ if for any $0 \leqslant i \leqslant d$ exactly $t_i$ of $x\text{'s}$ coordinates equal $i$.

    We will also often treat a type tuple $(t_0, \ldots, t_d)$ (for $t_0+\ldots+t_d=n$) as the set of vertices of the hypercube of such type.

    According to this convention, a type tuple $(t_0, \ldots, t_d)$ with $t_i < 0$ for some $i$ is just an empty set.
\leanblueprint{type}
\end{definition}

Two vertices are of the same type if and only if their coordinates are a permutation of each other (i.e., types are orbits of the coordinate permutation group).

There are $\binom{n}{t_0, \ldots, t_{d-1}}$ vertices of type $(t_0, \ldots, t_{d})$, where we define the multinomial symbol as:
\begin{equation*}
    \binom{n}{a_1, \ldots, a_k}:=\frac{n!}{a_1! \cdot \ldots \cdot a_k! \cdot (n-a_1-\ldots-a_k)!},
\end{equation*}
(if $a_i < 0$ for some $i$ or $n-a_1-\ldots-a_k<0$, then $\binom{n}{a_1, \ldots, a_k}=0$).

Analogously to the binomial coefficient, we have the following recursive formula (which may be easily proven using the combinatorial interpretation of the multinomial coefficient):
\begin{equation}
\label{multinomial_newton_triangle}
    \binom{n}{a_1, \ldots, a_k} = \binom{n-1}{a_1, \ldots, a_k} + \binom{n-1}{a_1-1, \ldots, a_k} + \ldots + \binom{n-1}{a_1, \ldots, a_k-1}.
\end{equation}

Clearly, all vertices of the same type belong to the same layer. For $d=1$, the type of a vertex is determined by its layer (all elements of $L_i$ must have exactly $i$ ones and $n-i$ zeroes, so are of type $(i, n-i)$). For $d>1$, most layers contain multiple types.

Similarly to calling layers lower/upper, outer/inner, we can call a type \emph{lower} if it belongs to a lower layer, etc.

\section{The proof strategy: weighted chains}
\label{section_assigning_preliminary}

This section discusses the general weighted chains strategy for proving Theorem \ref{the_main_theorem}. The results here hold for any $d \geqslant 1$.

The Proposition below is the key step in proving Theorem \ref{the_main_theorem}. After stating and discussing it we show how it implies the theorem, and in the subsequent sections provide a construction of weights satisfying the required properties.

We recall sets $A_1, A_2$ defined in Theorem \ref{the_main_theorem} (where $A_1=A_2$ if $2 \mid nd$).

\begin{proposition}
\label{wc:prop:good-chain-weighting}
\label{the_weights_assigning_proposition}
    For any $n, 1 < k \leqslant n$ and $d \in \{1, 2\}$, it is possible to assign weights to good chains in the hypercube so that:
    \begin{enumerate}
        \item \label{sum_1_condition} The sum of weights of chains passing through each vertex (which we will call the induced weight of that vertex) equals $1$.
        
        \item \label{non_negative_condition} The weights are non-negative.

        \item \label{short_chains_positive_condition} The weights of all the shortest chains (as defined in Equations \ref{singleton_chains_equation} and \ref{double_chains_equation}), except possibly for $\{\thespecialpoint\}$, are positive.
        
        \item \label{positive_non_A_condition} For each $i \in \{1, 2\}$ and $x \notin A_i$ there exists $y \in A_i \setminus \{\thespecialpoint\}$, $y$ closer to the middle of the cube than $x$ and a good chain of positive weight passing through $x$ and $y$. 
        
        \item \label{positive_A_condition} For any $x \in (A_1 \cup A_2) \setminus \{\thespecialpoint\}$ there exists a good chain of positive weight starting at $x$.
    \end{enumerate}
    In the penultimate condition, $y$ is closer to the middle of the cube than $x$ if and only if $\vert |y| - \frac{nd}{2} \vert < \vert |x| - \frac{nd}{2} \vert$.
    
    The last three conditions are only needed for showing the uniqueness of $A_1$ and $A_2$.
\leanblueprint{good-chain-weighting}
\end{proposition}

We note that at this stage it is by no means clear why the point $\thespecialpoint$ (the centre of the cube, if it belongs to the integer lattice) is treated separately in the last three conditions. The reason for this is that the most intuitive weight assignment we later construct assigns weight $0$ to this point, and so it is simplest to allow for this exception in the Proposition.

We also note that these conditions make sense when $2 \nmid d$, in which case simply $A_i \setminus \{\thespecialpoint\}$ equals $A_i$.

All the shortest chains (as specified in Equations \ref{singleton_chains_equation} and \ref{double_chains_equation}) are good chains. However, at first it is not clear whether for any $x \in A_1 \cup A_2$ there exists a good chain starting at $x$ and whether for any $x \notin A_i$ there exists a good chain passing through $x$ and through an element of $A_i$ closer to the middle of the cube than $x$.

For $d=1$ assigning positive weights to all good chains (which, as we show in Section \ref{section_d1} is possible) is sufficient: for any such assignment conditions $(\ref{positive_non_A_condition})$ and $(\ref{positive_A_condition})$ are satisfied as there exist good chains starting at any point (and as $(\frac{1}{2}, \dots, \frac{1}{2}) \notin \{0,1\}^n$ so no point requires a special treatment).

For $d=2$ (handled in Section \ref{section_d2}) first we show in Lemma \ref{positive_basic_enough_lemma} that conditions $(\ref{positive_non_A_condition})$ and $(\ref{positive_A_condition})$ hold not only for good chains, but also for their certain subset called `basic chains', which is a bit simpler and easier to analyse. Then we show that there exists an assignment of weights to basic chains which satisfies conditions (\ref{sum_1_condition}) - (\ref{short_chains_positive_condition}).

We note that the Proposition cannot hold for $d > 2$, as in this case there exist points in the cube at which no good chains start. For example, let us consider an even $d=2d'>2$, any $k$ and $n$ large enough divisible by $(dk+1)$. Let $x$ be a point with all coordinates equal $d'-1$. Then $x$ is a lower point and since $dk+1 \mid n$, $|x|=n(d'-1)$, which is equivalent to $nd' = \fclass$ modulo $dk+1$. Then $x \in A_1=A_2$ but there is no chain of width $k$ and length $kd+1$ starting at $x$ (as $x$ has no coordinate $0$) and also, for $n$ large enough, no symmetric chain starting at $x$ and of width at most $k$. Hence, there are no good chains starting at $x$, and so condition $(\ref{positive_A_condition})$ cannot be satisfied. We could analogously construct counterexamples for odd $d$.

Let us now show how the Proposition implies Theorem \ref{the_main_theorem}.

\begin{lemma}
\label{wc:lem:weighted-cover-implies-main}
\label{weights_imply_theorem_lemma}
    For any $n$, $1 < k \leqslant n$ and $d \geqslant 1$ if there exists an assignment of weights to good chains in $\thehypercube$ satisfying conditions (\ref{sum_1_condition}) and (\ref{non_negative_condition}), then any $k$-separated set in the hypercube is at most as large as $A_1$ (for $A_1$ defined as in Theorem \ref{the_main_theorem}).

    Furthermore, if such an assignment satisfies conditions (\ref{short_chains_positive_condition}), (\ref{positive_A_condition}) and (\ref{positive_non_A_condition}) of the Proposition, then $A_1$ and $A_2$ are the unique largest such subsets (and so Theorem \ref{the_main_theorem} is satisfied).
\leanblueprint{weighted-cover-implies-main}
\end{lemma}

\begin{proof}
    First let us show the first part of the lemma, and then the uniqueness part.

    Let $w: \mathcal{C} \rightarrow [0, 1]$ be an assignment of non-negative weights to the family $\mathcal{C}$ of good chains, satisfying the conditions of Proposition \ref{the_weights_assigning_proposition}. Hence, for $x \in \thehypercube$, we have:
    \begin{equation*}
        \sum_{C \in \mathcal{C}: \ x \in C} w(C) = 1.
    \end{equation*}
    Since each good chain contains exactly one element of $A_1$, we have:
    \begin{equation}
        \vert A_1\vert  = \sum_{x \in A_1} 1 = \sum_{x \in A_1} \sum_{C \in \mathcal{C} : x \in C} w(C) = \sum_{C \in \mathcal{C}} w(C) \cdot \vert A_1 \cap C\vert  = \sum_{C \in \mathcal{C}} w(C).
    \end{equation}
    On the other hand, let $B \subseteq \thehypercube$ be any $k$-separated set. Then, since any good chain $C$ may contain at most one element of $B$, so similarly to the above calculations, we have:
    \begin{equation}
    \label{chain_weights_bound_the_size}
        \vert B\vert  = \sum_{x \in B} 1 = \sum_{x \in B} \sum_{C \in \mathcal{C} : x \in C} w(C) = \sum_{C \in \mathcal{C}} w(C) \cdot \vert B \cap C\vert  \leqslant \sum_{C \in \mathcal{C}} w(C).
    \end{equation}
    Hence, we must have $\vert B\vert  \leqslant \vert A_1\vert =\vert A_2\vert $.

    Now let us show that there are no other sets $A'$ satisfying the conditions of Theorem \ref{the_main_theorem} with $\vert A'\vert =\vert A_1\vert =\vert A_2\vert $. The main idea of the proof is quite simple, but the special treatment of $\thespecialpoint$ requires some extra attention to details.

    For any $C \in \mathcal{C}$ with positive weight, such $A'$ must contain at least one (and so exactly one) element of $C$, as otherwise we could bound its size as in Equation \ref{chain_weights_bound_the_size}, with a strict inequality at the end.

    Let us consider the implications of condition (\ref{short_chains_positive_condition}) separately for the cases $2 \mid nd$ (so $A_1=A_2$) and $2 \nmid nd$ (so $A_1, A_2$ are distinct).

    \begin{caseindent}
    $1: \ 2 \mid nd.$ Since the weights of all singleton chains except for $\{\thespecialpoint\}$ are positive, $x \in A'$ for any $x \in L_{\frac{nd}{2}}$ (except possibly for $x=\thespecialpoint$).

    $2: \ 2 \nmid nd.$ For any $x, y$ with $\vert x\vert =\fclass$, $\vert y\vert =\cclass$ and $x \leqslant y$ the chain $\{x, y\}$ has positive weight, so exactly one of $x$ and $y$ is in $A'$.
    
    We may draw this condition on a bipartite graph, whose vertices are points from the layers $L_{\fclass}$ and $L_{\cclass}$ and where two vertices $x \in L_{\fclass},$ $y \in L_{\cclass}$ are connected by an edge if and only if $x \leqslant y$. Since this bipartite graph is easily seen to be connected, the only subsets of vertices of this graph which contain exactly one vertex from each edge are $L_{\fclass}$ and $L_{\cclass}$.

    Hence, we need either $L_{\fclass} \subseteq A'$ or $L_{\cclass} \subseteq A'$. Without loss of generality, let us assume that the former is the case (and then $L_{\cclass} \cap A'=\emptyset$).
    \end{caseindent}
    
    We will now show (simultaneously for both cases) that we must have $A'=A_1$.

    For $0 \leqslant i \leqslant \cclass$ let $B_i$ denote a `slice' of the cube around the middle of thickness proportional to $i$, without $\thespecialpoint$ - that is $B_i:=\{x: \big\vert \vert x\vert -\frac{nd}{2}\big\vert  \leqslant i + \frac{1}{2}\} \setminus \{\thespecialpoint\}$. We have shown that $B_0 \cap A'=B_0 \cap A_1$. Now we will show that in fact, for any $i$ we have $B_i \cap A'=B_i\cap A_1$.

    Let us suppose that $B_{i-1} \cap A' = B_{i-1} \cap A_1$. We'd like to show that the same holds for $i$.

    \begin{caseindent}
        $1:$ First, let us consider the case when $i$ is such that neither $L_{\fclass - i}$ nor $L_{\cclass + i}$ is a subset of $A_1$. But then for any $x \in B_i \setminus B_{i-1} = L_{\fclass - i} \cup L_{\cclass + i}$, from condition (\ref{positive_non_A_condition}) there is a chain of positive weight passing through $x$ and through some $y \in A_1 \setminus \{\thespecialpoint\}$, $y$ closer to the middle of the hypercube than $x$ (that is, $y \in B_{i-1}$). But then by induction, $y \in A'$, and so we must have $x \notin A'$. Applying this reasoning to every $x \in L_{\fclass - i} \cup L_{\cclass + i}$, we conclude that $B_i \cap A'=B_i \cap A_1$.

        $2:$ Now, let us consider $i$ such that $L_{\fclass - i} \subseteq A_1$, $L_{\cclass + i} \not\subseteq A_1$ (the case when $L_{\fclass - i} \not\subseteq A_1$ and $L_{\cclass + i} \subseteq A_1$ can be considered analogously). Then for any $x \in L_{\fclass - i}$, there exists a chain of positive weight starting at $x$. But this chain can only (except for $x$) pass through elements not belonging to $A_1$ that are closer to the middle of the hypercube than $x$ (so are in $B_{i-1}$), and so by induction not belonging to $A'$. Hence, we must have $x \in A'$. Moreover, just as in case $1$ for any $x \in L_{\cclass + i}$ we have $x \not\in A'$. 

        $3:$ Finally, let us suppose that $L_{\fclass - i} \subseteq A_1$ and $L_{\cclass + i} \subseteq A_1$. Since also $L_{\fclass} \subseteq A_1$, from the definition of $A_1$ it follows that $dk+1 \mid i$ and $2 \mid nd$ (i.e., $\fclass=\cclass$). However, then the distance between $L_{\fclass - i}$ and $L_{\cclass + i}$ is at least $2(dk+1)$ and so there is no good chain passing through elements of both these layers. A reasoning analogous to that in case $2$ leads us to a conclusion that $L_{\fclass - i}, L_{\cclass + i} \subseteq A'$ and so again $B_i \cap A'=B_i \cap A_1$. \qedhere
        \end{caseindent}
\end{proof}

In the rest of the paper, we prove Proposition \ref{the_weights_assigning_proposition} by constructing a suitable assignment of weights.

We should aim at an assignment of weights to chains satisfying the following two conditions:
\begin{itemize}
    \item It should be invariant under a permutation of coordinates.
    \item It should be symmetric, i.e. the chains $(x_1, \ldots, x_l)$ and $(\textbf{d}-x_l, \ldots, \textbf{d}-x_1)$ should have equal weights (where $\textbf{d}=(d, \ldots, d)$).
\end{itemize}

The first of these conditions is equivalent to saying that all chains passing through exactly the same types should have equal weights.

In the following two sections, we will define an assignment of weights to good chains satisfying the conditions of Proposition \ref{the_weights_assigning_proposition} (and prove its desired properties) first for $d=1$ and then for $d=2$.  In both cases, we will (recursively) assign weights to chains that clearly satisfy condition (\ref{sum_1_condition}) and then show that this assignment also satisfies conditions (\ref{non_negative_condition})-(\ref{positive_A_condition}).

\section{Assigning weights for \texorpdfstring{$d=1$}{d=1}}
\label{section_d1}

\subsection{Assigning weights and showing that chains starting at the outer layers have positive weights}

We will assign weights to saturated (see Definition \ref{saturated_definition}) chains that are either of length exactly $k+1$ (and so of width $k$) or symmetric and shorter than $k+1$. As in Definition \ref{good_chain_definition}, let us call such chains \emph{good}.

Figures \ref{d1_drawing1} and \ref{d1_drawing2} below illustrate good chains in two representative cases. On each, we have a number of dots, representing layers $L_0, L_1, \dots, L_n$. The black dots represent the middle layer (for $2 \mid n$) or layers (for $2 \nmid n$).

The braces show through which layers good chains pass. For example, a brace between the second and fifth upper dots on the first picture represents all chains passing exactly through layers $L_1, L_2, L_3, L_4$.
The braces on the left-hand side of the dots represent the good chains of length exactly $k+1$, and those on the right-hand side represent the chains of length shorter than $k+1$, but that are symmetric (and so also good).

\begin{figure}[h!]
    \centering
    \begin{minipage}{0.45\textwidth}
        \centering
        \includegraphics[scale=0.7]{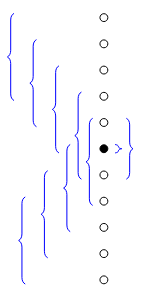}
        \caption{$n=10, k=3$}
        \label{d1_drawing1}
    \end{minipage}\hfill
    \begin{minipage}{0.45\textwidth}
        \centering
        \includegraphics[scale=0.55]{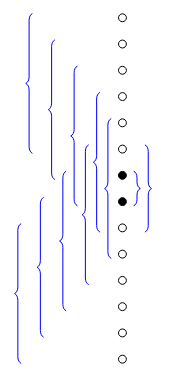}
        \caption{$n=13, k=5$}
        \label{d1_drawing2}
    \end{minipage}
\end{figure}

All good chains starting at the same layer have the same length (depending on how far from the middle of the cube that layer is) and pass through the same layers. Moreover, from the symmetry of permuting coordinates, the same number of good chains start at each element of a given layer, and for any layers $L_i, L_j$ exactly the same number of good chains starting at $L_i$ pass through each element of $L_j$.

Hence, if we assign equal weights to all good chains starting at the same layer, all vertices from the same layer will have the same induced weight.

For a given $n$, let $W_n(a)$ denote the total weight assigned to all good chains starting at layer $a$.

Starting at $a=0$ and proceeding through $a=1, \ldots, \big\lfloor \frac{n}{2} \big\rfloor$, we will recursively choose $W_n(a)=W_n(n-a)$ so that the total weight of all chains passing through $L_a$ (and $L_{n-a}$) is $\binom{n}{a}$ (so that the induced weight of each vertex is $1$). By symmetry, this will ensure that the weight of each vertex from the upper side of the cube is also $1$. This ensures condition (\ref{sum_1_condition}) of Proposition \ref{the_weights_assigning_proposition}.

Then, we will verify that all weights $W_n(a)$ assigned this way are positive. This clearly ensures conditions (\ref{non_negative_condition}) and (\ref{short_chains_positive_condition}) of the Proposition. Since we may easily check that for any point in $\{0, 1\}^n$ there exists a good chain starting at it (and any good chain passes through a point from $A_1$ and a point from $A_2$), this also ensures that conditions (\ref{positive_non_A_condition}) and (\ref{positive_A_condition}) hold.

We must have $W_n(0)=1$ (as there is exactly one vertex in layer $0$, $(0, \ldots, 0)$). To assign weights to chains starting at further layers, we note that (for $a \leqslant \big\lfloor \frac{n}{2} \big\rfloor$, so $L_a$ in the lower side of the cube) the chains passing through $L_a$ are exactly the chains passing through $L_{a-1}$, except for:
\begin{itemize}
    \item The chains starting at $L_a$.
    \item The chains starting at $L_{a-k-1}$ (as they finish exactly at $L_{a-1}$).
    \item If $L_a$ is an inner layer (we defined inner and outer layers in Definition \ref{wc:def:inner-outer-layers}), the chains starting at $L_{a+k}$ pass through $L_a$ but not through $L_{a-1}$.
\end{itemize}
Hence, for $a$ a lower outer layer the difference between the total weight of chains passing through $L_a$ and $L_{a-1}$ can be written in two ways as:
\begin{equation*}
    \wn(a) - \wn(a-k-1) = \binom{n}{a} - \binom{n}{a-1},
\end{equation*}
and so we set
\begin{equation}
\label{d1_w_equation_outer}
    W_n(a) = \binom{n}{a} - \binom{n}{a-1} + W_n(a-k-1).
\end{equation}
Similarly, for $a$ a lower inner layer, we need to have:
\begin{equation*}
    \wn(a) - \wn(a-k-1) + \wn(a+k) = \binom{n}{a} - \binom{n}{a-1},
\end{equation*}
and so
\begin{equation}
\label{d1_w_equation_inner}
    W_n(a) = \binom{n}{a} - \binom{n}{a-1} + W_n(a-k-1) - W_n(a+k).
\end{equation}
This ensures that for ny $a$ the total weight of all chains passing through layer $L_a$ is $\binom{n}{a}$, and so that for any $x \in L_a$, the total weight of all chains passing through $x$ is $1$.

We can also easily see that the weights of the chains starting at the outer layers are positive, since (for $a \leqslant \frac{n}{2}$) we have $\binom{n}{a}-\binom{n}{a-1}>0$ and so, inductively, $W_n(a-k-1) > 0$.

Now it only remains to show that the weights of chains starting at the inner layers (i.e. the symmetric chains shorter than $k+1$) are positive. It is a slightly harder task than for the outer layers, as there might be chains from the other side of the cube passing through these layers.

\subsection{Showing that chains starting at the inner layers have positive weights}

For $a \leqslant \frac{n}{2}$, the layer $a$ is inner if and only if $n-2a+1$ (the length of a symmetric chain starting at $a$) is smaller than $k+1$, that is if $n-k<2a$.

We note that it is possible to do it in a more direct way than presented here, but the proof presented below adapts more naturally to $d=2$. 

First, let us define a function $U_n$ inductively as follows:
$U_n(a)=0$ for $a<0$, $U_n(0)=1$ and
\begin{equation}
    U_n(a) = \binom{n}{a} - \binom{n}{a-1} + U_n(a-k-1)
\end{equation}
for $a>0$.

Clearly, for $a$ a lower outer layer, we have $W_n(a)=U_n(a)$ (as they satisfy the same base conditions and inductive equation). Moreover, by an argument analogous to the one for $W_n$, $U_n(a)>0$ for any $0 \leqslant a \leqslant \frac{n}{2}$.

From the symmetry, for $a$ an upper outer layer we have $W_n(a)=W_n(n-a)=U_n(n-a)$. Finally, since $U_n$ is defined only in terms of binomial coefficients and other values of $U_n$, we may inductively show that for any $0 \leqslant a \leqslant n$ we have
\begin{equation}
    U_n(a)=\uno(a)+\uno(a-1).
\end{equation}

\begin{lemma}
\label{wc:lem:boolean-inner-weight}
     For $a$ a lower inner layer, we have
    \begin{equation}
        W_n(a)=U_n(a) - U_n(n-a-k).
    \end{equation}
\leanblueprint{boolean-inner-weight}
\end{lemma}

\begin{proof}
    For $a$ a lower inner layer both $a-k-1$ and $a+k$ are outer layers. Hence, $\wn(a-k-1)=\un(a-k-1)$ and $\wn(a+k)=\wn(n-a-k)=\un(n-a-k)$ and so (using \ref{d1_w_equation_inner}) we have:
    \begin{equation*}
        \wn(a) = \binom{n}{a} - \binom{n}{a-1} + \un(a-k-1) - \un(n-a-k) = \un(a) - \un(n-a-k)
    \end{equation*}
    as required.
\end{proof}

Now we are ready to show that $W_n(a)>0$ for lower inner layers.

\begin{lemma}
\label{wc:lem:boolean-inner-weight-positive}
\label{d1_lower_inner_positive}
    For $1 < k \leqslant n$ and $n-k < 2a \leqslant n$ we have $U_n(a) - U_n(n-a-k) > 0.$
\leanblueprint{boolean-inner-weight-positive}
\end{lemma}

\begin{proof}
Somewhat surprisingly, we will fix $k$ and show this by induction on $n$. Since we are interested in $k \leqslant n$, the base case of the induction is $n=k$.

$1:$ For $n=k$ and $0<2a \leqslant n$, we have $W_n(a)=U_n(a)-U_n(-a)$. Since $a>0$, $U_n(-a)=0$ and from the previous discussion we have $U_n(a)>0$, which finishes the proof in this case.

$2:$ Let $n>k$. We have
\begin{equation*}
    U_n(a)=U_{n-1}(a)+U_{n-1}(a-1)
\end{equation*}
and
\begin{equation*}
    U_n(n-a-k)=U_{n-1}(n-a-k)+U_{n-1}(n-a-k-1).
\end{equation*}
\begin{caseindent}
    $2.1:$ If we have
    \begin{equation}
        (n-1)-k<2a\leqslant n-1
    \end{equation}
    and
    \begin{equation}
        (n-1)-k<2(a-1)\leqslant n-1,
    \end{equation}
    then $a$ and $a-1$ are lower inner layers of $\{0,1\}^{n-1}$ and so
    \begin{equation*}
        \wno(a)=\uno(a)-\uno((n-1)-a-k)
    \end{equation*}
    and
    \begin{equation*}
         \wno(a-1)=\uno(a-1)-\uno((n-1)-(a-1)-k).
    \end{equation*}
    Hence
    \begin{equation*}
        W_n(a)=W_{n-1}(a)+W_{n-1}(a-1),
    \end{equation*}
    which by induction is a sum of two positive integers.
    
    For $n-k<2a \leqslant n$, we always have $(n-1)-k<2a$ and $2(a-1)\leqslant n-1$, and at most one of the other two inequalities might be broken.
    
    $2.2:$ If $2a > n-1$, then (since also $2a \leqslant n$), we must have $2a=n$.
    Then $U_n(a)=U_{n-1}(a)+U_{n-1}(a-1)$ and
    \begin{equation*}
        U_{n-1}(a)=\binom{n-1}{a}-\binom{n-1}{a-1}+U_{n-1}(a-k-1)=U_{n-1}(a-k-1).
    \end{equation*}
    We also have
    \begin{equation*}
        U_n(n-a-k)=\uno(n-a-k)+\uno(n-a-k-1)=\uno((n-1)-(a-1)-k)+\uno(a-k-1).
    \end{equation*}
    Hence,
    \begin{equation*}
        W_n(a)=U_n(a)-U_n(n-a-k)=\uno(a-1)-\uno((n-1)-(a-1)-k)=W_{n-1}(a-1),
    \end{equation*}
    which by induction is positive (since $(n-1)-k<2(a-1)=n-2 \leqslant n-1$, using $k>1$, so $a-1$ is a lower inner layer of $\{0,1\}^{n-1}$).
    
    $2.3:$ If $(n-1)-k \geqslant 2(a-1)$, then (since also $n-k<2a$), we must have $2a+k-1=n$. But then $U_{n-1}(a-1)=U_{n-1}(n-a-k)$ and hence
    \begin{multline*}
        W_n(a)=U_n(a) - U_n(n-a-k)\\
        =(\uno(a) + \uno(a-1))-(\uno(n-a-k) + \uno(n-a-k-1))\\
        =U_{n-1}(a)-U_{n-1}((n-1)-a-k)=W_{n-1}(a),
    \end{multline*}
    which by induction is positive (since $(n-1)-k<2a \leqslant n-1$, so $a$ is a lower inner layer of $\{0, 1\}^{n-1}$). \qedhere
\end{caseindent}
\end{proof}

\section{Assigning weights for \texorpdfstring{$d=2$}{d=2}}
\label{section_d2}

Although the structure of the hypercube for $d=2$ is more complex (for example, not all elements from the same layer are in the same index permutation orbit), the main outline of the proof and ideas follow from trying to generalise the proof presented above for $d=1$.

\subsection{Assigning weights}
\label{d2_weights_assignement_subsection}

It suffices to assign weights to good chains in the hypercube in a way satisfying the conditions of Proposition \ref{the_weights_assigning_proposition}. We'd like to do this analogously to the natural assignment of weights for $d=1$, i.e. considering vertices from the furthest ones from the middle towards the middle, at each step assigning (for symmetry) the same weight to all chains starting at the same vertex (the weight being determined by the requirement that the sum of weights of chains passing through any vertex is $1$).

However, after running a computer simulation, it turns out that this recursive assignment of weights would result in negative weights being assigned.

Hence, instead, we will do the same simple approach, but assign positive weights only to a subset of all good chains.

\begin{definition}[Basic chains]
\label{wc:def:basic-chain}
    Let $(x_1, \ldots, x_l)$ be a saturated chain of width $w$ and length $l$. We call this chain basic if $l=dw+1$ and its subchains $(x_{id+1}, x_{id+2}, \ldots, x_{id+d+1})$ (for $0 \leqslant i \leqslant w-1$) are all of width $1$.
\leanblueprint{basic-chain}
\end{definition}

For example, for $n=5, d=2$ the following chain is basic:
\begin{equation*}
    \{(0, 2, 1, 0, 0), (0, 2, 1, 0, 1), (0, 2, 1, 0, 2), (1, 2, 1, 0, 2), (2, 2, 1, 0, 2)\},
\end{equation*}
as its subchains
\begin{equation*}
    \{(0, 2, 1, 0, 0), (0, 2, 1, 0, 1), (0, 2, 1, 0, 2)\}
\end{equation*}
and
\begin{equation*}
    \{(0, 2, 1, 0, 2), (1, 2, 1, 0, 2), (2, 2, 1, 0, 2)\}
\end{equation*}
have only one coordinate varying from $0$ to $2$.

However, the following chain is not basic:
\begin{equation*}
    \{(0, 2, 1, 0, 0), (0, 2, 1, 0, 1), (1, 2, 1, 0, 1), (2, 2, 1, 0, 1), (2, 2, 1, 0, 2) \},
\end{equation*}
as it first changes on coordinate $5$, then on coordinate $1$ (twice) and then again on coordinate $5$.

We note that for $d=1$ all saturated chains are basic.

Basic chains that are also good are those basic chains that are either of width $k$ (and so starting at an outer layer $a$) or of width $w<k$ and `between' layers $L_{n-w}$ and $L_{n+w}$.

In what follows, we will show that (for $d=2$) it is possible to assign weights to basic good chains so as to satisfy the conditions of Proposition \ref{the_weights_assigning_proposition}.

\begin{remark}
    At first sight it is by no means obvious that taking only basic good chains should be enough. We noticed this after making simulations in Python. We also tried some other subsets of chains, which turned out to require negative weights.
\end{remark}

Let us recall Definition \ref{wc:def:type}, where we defined types. We note that for $d=2$, the type $(a, b, c)$ lies in layer $b+2c=n-a+c$, and $L_n$ is the middle layer of the hypercube (hence, a type $(a, b, c)$ is lower if $a \geqslant c$ and upper if $a \leqslant c$).

A useful property of the set of basic chains is that all basic chains of a given width $w$ starting at vertices of the same type (e.g. $(a, b, c)$ with $a>c$ for $d=2$) pass through exactly the same types, e.g.
\begin{equation*}
    (a, b, c), (a-1, b+1, c), (a-1, b, c+1), (a-2, b+1, c+1), (a-2, b, c+2), \ldots, (a-w, b, c+w).
\end{equation*}
Hence, the set of all basic good chains starting at a given type is one of the orbits of the action of permuting coordinates on the set of chains. Therefore, it will be natural to assign the same weights to all basic good chains starting at the same type.

Given the symmetry (under permuting coordinates) of all vertices in the same type (and all basic good chains), we will think of all basic good chains starting at vertices from the same type together, as one \emph{metachain}.

It may be helpful to visualise the types of $\{0, 1, 2\}^n$ on a \emph{staircase plot}, as shown in Figure \ref{figure_chain}. It shows the unit squares in the coordinate system (with the horizontal axis labelled as $a$ and the vertical axis labelled as $c$) whose coordinates satisfy $0 \leqslant a, c$ and $a+c \leqslant n$ (hence a staircase shape).

A square at position $(a, c)$ represents all points of type $(a, n-a-c, c)$ in $\{0, 1, 2\}^n$.

Then for example the types from the same layer are those lying on diagonal lines with $a-c$ constant (in particular, the squares with $a=c$ represent the types from the middle layer $L_n$).

For a given $k$, we can also represent all basic chains starting at a lower outer type $(a, n-a-c, c)$ as shaded squares, as shown in the figure. For $k=2$, the following squares represent types that chains starting at type $(a, n-a-c, c)$ pass through: $(a, c)$, the square left of it $(a-1, c)$, the square above it $(a-1, c+1)$, the square left of it $(a-2, c+1)$ and the square above it $(a-2, c+2)$.

\begin{figure}
    \centering
    \includegraphics[width=0.45\linewidth]{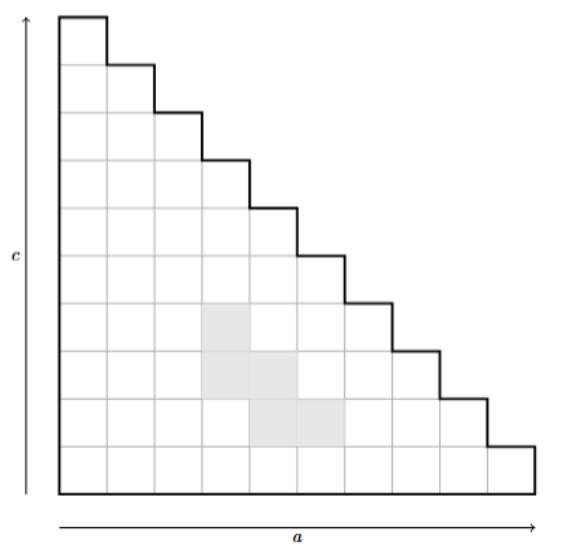}
    \caption{$n=9, k=2$}
    \label{figure_chain}
\end{figure}

It turns out that for $d=2$ it is possible to assign positive weights only to basic good chains and weight $0$ to all other chains so that the conditions of Proposition \ref{the_weights_assigning_proposition} are satisfied. In the lemma below we show that an assignment of positive weights to all basic good chains (except for $\{(1, \dots, 1)\}$) would satisfy conditions (\ref{non_negative_condition})-(\ref{positive_A_condition}) of the Proposition.

\begin{lemma}
\label{wc:lem:basic-chains-suffice}
\label{positive_basic_enough_lemma}
    As in Theorem \ref{the_main_theorem}, let $A_1$ denote the points $x \in \{0, 1, 2\}^n$ with $|x| \equiv n \ (\modd 2k+1)$.

    The singleton chains $\{x\}$ for $x \in L_{n}$ are basic.

    Moreover, for any $x \in A_1$, there exists a basic good chain starting at $x$.

    Finally, for any $x \notin A_1$, there exists $y \in A_1 \setminus \{(1, \dots, 1)\}$, $y$ closer to the middle of the cube than $x$, and a basic good chain passing through $x$ and $y$.
\leanblueprint{basic-chains-suffice}
\end{lemma}

\begin{proof}
    The first part follows directly from the definition of basic chains.

    Since any good chain passes through an element of $A_1$, and the symmetric chains pass through the middle layer of the cube, to prove the second and third parts it suffices to show that any vertex $x \neq \theonepoint$ has a chain starting at it which does not pass through $\theonepoint$. We will show that this holds for all vertices $x \in A_1$ and for most vertices $x \notin A_1$, and for those vertices for which it fails we will point to other chains passing through it which satisfy the required condition.

    From the symmetry, it suffices to show this for lower vertices.

    For outer vertices the claim is clear: if a lower outer vertex $x$ is of type $(a, b, c)$ (that is, $c<a-k$), then we have $k \leqslant c+k<a$, and so $x$ has more than $k$ coordinates equal to $0$ and has a basic chain of length $kd+1$ starting at it (which, if $x \notin A_1$, must pass through an element of $A_1$ with at least one coordinate $0$, so different from $(1, \dots, 1)$).

    For $x$ inner of type $(a, b, c)$, $x \in L_l$ (where $l=b+2c$), there are symmetric chains of width $n-l$ and length $2n-2l+1$ starting at $x$. If $c>0$, all the vertices they pass through have at least one coordinate equal to $2$, and so they do not pass through $\theonepoint$.
    
    Hence, the only types that require careful checking are $(a, n-a, 0)$ for $a < k$. If $a=0$, the type is $(0, n, 0)$ which contains only $\theonepoint$.
    
    If $a>1$, $2 \mid a$, the basic good chains starting at vertices of this type pass through a vertex of type $(\frac{a}{2}, n-a, \frac{a}{2})$, which is in the middle layer $L_n$ and so in $A_1$. If $a>1$, $2 \nmid a$, such chains pass through a vertex of type $(\frac{a-1}{2}, n-a+1, \frac{a-1}{2})$, which is in $L_n \subseteq A_1$.

    Finally, the only vertices left to consider are those of type $(1, n-1, 0)$ (which are not in $A_1$ for $k>0$). However, for any of them there is a chain (actually, multiple chains) starting at a vertex of type $(2, n-2, 0)$ and passing through a vertex from the middle layer of type $(1, n-2, 1)$.
\end{proof}

Similarly to the case $d=1$, we will assign weights to basic good chains, first to those starting at types $(n, 0, 0)$ and $(0, 0, n)$, and then consider types increasingly closer to the middle of the cube. At each point we will assign the same weight to all basic good chains starting at a given type, so that the total weight of all chains passing through that type equals its size (as then, from the symmetry of vertices of the same type, the weight of all chains passing through any vertex of that type is $1$).

This ensures condition (\ref{sum_1_condition}) of Proposition \ref{the_weights_assigning_proposition}, and leaves us to show that conditions (\ref{non_negative_condition})-(\ref{positive_A_condition}) are satisfied. We will do this at once by showing that the weight of all basic chains is positive, except for the singleton chain $\{(1, \dots, 1)\}$, which has weight $0$. By Lemma \ref{positive_basic_enough_lemma}, this implies that conditions (\ref{non_negative_condition})-(\ref{positive_A_condition}) of the Proposition are satisfied.

The total weight of all basic good chains passing through type $(a, b, c)$ needs to be $\binom{n}{a, c}$.

Let us denote by $W_n(a, c)$ the total weight of basic good chains starting at type $(a, n-a-c, c)$ (their common weight times their number).

We will first show that the weights of chains starting at the outer types are positive and then that the weights of chains starting at the inner types are positive. We will show this for vertices from the lower half of the cube, as thanks to the symmetry the same will apply to the upper half of the cube. Hence, from now on $(a, b, c)$ will denote a lower type (i.e. with $a \geqslant c$).

It may be easily checked that a lower type is outer if and only if $c+k \leqslant a$ and inner otherwise. Moreover, as mentioned above, the vertices of type $(a, b, c)$ lie in layer $b+2c=n-a+c$. Finally, let us recall that a type $(a, b, c)$ with any of $a, b$ or $c$ negative is just an empty set.

The key observation for calculating the weights of chains is that basic good chains pass through a lower type $(a, b, c)$ if and only if they pass through type $(a+1, b, c-1)$, except for:
\begin{itemize}
    \item The chains starting at $(a, b, c)$ and $(a+1, b-1, c)$.
    \item The chains starting at $(a+k+1, b, c-k-1)$ and $(a+k+1, b-1, c-k)$, as they end `just before' $(a, b, c)$ (they are of length $2k+1$ and finish respectively at types $(a+1, b, c-1)$ and $(a+1, b-1, c)$).
    \item For $(a, b, c)$ a lower inner type, additionally chains starting at $(a-k, b, c+k)$ pass through $(a, b, c)$ but not through $(a+1, b, c-1)$.
    \item For $(a, b, c)$ a lower inner type with $c > a-k+1$, also chains starting at $(a-k+1, b-1, c+k)$ pass through $(a, b, c)$ but not $(a+1, b, c-1)$. (For $c=a-k+1$, the good chains starting at $(a-k+1, b-1, c+k)$ are symmetric of length $2k+1$ and pass through both $(a, b, c)$ and $(a+1, b, c-1)$.)
\end{itemize}

The difference between the sum of weights of chains passing through types $(a, b, c)$ and $(a+1, b, c-1)$ should be $\binom{n}{a, c}-\binom{n}{a+1, c-1}$. From the above we also know that for lower outer types it equals
\begin{equation*}
    \wn(a, c) + \wn(a+1, c) - \wn(a+k+1, c-k) - \wn(a+k+1, c-k-1)
\end{equation*}
and so for these types we have:
\begin{multline}
\label{the_key_recursive_equation}
    W_n(a, c) = \binom{n}{a, c} - \binom{n}{a+1, c-1} - W_n(a+1, c)\\ + W_n(a+k+1, c-k) + W_n(a+k+1, c-k-1).
\end{multline}
We can similarly derive the equation for the types $(a, b, c)$ with $c=a-k+1$ (in the `lowest' lower inner layer):
\begin{multline}
\label{the_key_recursion_lowest_inner}
    W_n(a, c) = \binom{n}{a, c} - \binom{n}{a+1, c-1} - W_n(a+1, c)\\ + W_n(a+k+1, c-k) + W_n(a+k+1, c-k-1) - W_n(a-k, c+k)
\end{multline}
and (for $c>a-k+1$):
\begin{multline}
\label{the_key_recursion_inner}
    W_n(a, c) = \binom{n}{a, c} - \binom{n}{a+1, c-1} - W_n(a+1, c)\\ + W_n(a+k+1, c-k) + W_n(a+k+1, c-k-1) - W_n(a-k, c+k) - \wn(a-k+1, c+k).
\end{multline}

Of course, $W_n(n, 0)=1$ (as there is exactly one vertex of type $(n, 0, 0)$) and $W_n(a, c)=0$ if $a<0, c<0$ or $n-a-c<0$. These starting conditions and Equations \ref{the_key_recursive_equation}, \ref{the_key_recursion_lowest_inner} and \ref{the_key_recursion_inner} allow us to recursively compute $W_n(a, c)$ for lower types, i.e. with $a \geqslant c$. For upper types (i.e. with $a \leqslant c$) by symmetry we have $W_n(a, c)=W_n(c, a)$.

Finally, let us introduce an auxiliary function $U_n$ (for $n \geqslant k$) defined as follows:
\begin{itemize}
    \item $U_n(a, c)=0$ if $a<0, c<0$ or $n-a-c<0$.
    \item $U_n(n, 0)=1$.
    \item For any other $a, c$,
    \begin{multline}
        U_n(a, c) = \binom{n}{a, c} - \binom{n}{a+1, c-1} - U_n(a+1, c)\\ + U_n(a+k+1, c-k) + U_n(a+k+1, c-k-1).
    \end{multline}
\end{itemize}

We note that $U_n$ has the same base values and recursive formula as $W_n(a, c)$ for $c+k \leqslant a$ (i.e., lower outer types), so for these types we have $U_n(a, c)=W_n(a, c)$.

For the upper outer types, we have $W_n(a, c)=W_n(c, a)=U_n(c, a)$.

Since $U_n$s are defined recursively in terms of each other and the trinomial coefficients, it may be easily shown by induction (using the property of multinomial coefficients expressed in Equation \ref{multinomial_newton_triangle}) that
\begin{equation}
\label{d2_u_n1}
    U_n(a, c)=\uno(a-1, c)+\uno(a, c)+\uno(a, c-1).
\end{equation}

We note that for $b=n-a-c$, the three $\uno$s above correspond to types $(a-1, b, c)$, $(a, b-1, c)$ and $(a, b, c-1)$.

\subsection{Showing that chains starting at the outer types have positive weights}
\label{d2_outer_positive}

Let us call all the types in $\{0, 1, 2\}^n$, except for $(0, n, 0)$ (which is the singleton type with point $(1, \dots, 1)$), \emph{proper} types. We will call the type $(0, n, 0)$ improper.

In Lemma \ref{wc:lem:ternary-auxiliary-positive} We will show that for $(a, b, c)$ a proper lower type (i.e. with $n=a+b+c$, $0 \leqslant c \leqslant a,$ $0 < a+c \leqslant n$), $U_n(a, c)$ is positive. Since for lower outer types $W_n(a, c)=U_n(a, c)$ and for upper outer types $W_n(a, c)=U_n(c, a)$, this immediately implies that the weights of chains starting at the outer types are positive. Before we state and prove the lemma, let us first discuss the proof strategy.

From Equation \ref{d2_u_n1}, for any triple $(n, a, c)$ if we know that $\uno(a-1, c), \uno(a, c)$ and $\uno(a, c-1)$ are all positive, then clearly $\un(a, c)$ is also positive. We may hope that the three $\uno$ terms can be naturally assumed positive by induction.

We will induct on $n$, treating $n=k$ as the base case. For most proper lower types $(a, b, c)$ in $\{0, 1, 2\}^n$, the types $(a-1, b, c)$, $(a, b-1, c)$ and $(a, b, c-1)$ are proper lower types in $\{0, 1, 2\}^{n-1}$, and the inductive step is trivial. However, for some types $(a, b, c)$ one or more of the three types might no longer be lower, or not proper, or might be a trivial empty type. We will need to check for this carefully, but each case in itself will be straightforward to handle.

The three types $(a-1, b, c)$, $(a, b-1, c)$ and $(a, b, c-1)$ are all lower proper if all the following conditions are met:
\begin{itemize}
    \item $0 \leqslant a-1, b-1, c-1$. Otherwise, if $a$, $b$ or $c$ is $0$ then respectively $a-1$, $b-1$ or $c-1$ equals $-1$, and so $(a-1, b, c)$, $(a, b-1, c)$ or $(a, b, c-1)$ is empty (and then $\uno$ of this type is $0$).
    \item $a, b, c \leqslant n$. Otherwise, similarly to above, we get an empty type.
    \item $1 < a+c$. Otherwise, if $a+c=1$, either $(a-1, b, c)$ or $(a, b, c-1)$ is the improper type $(0, n-1, 0)$.
    \item $c \leqslant a-1$. Otherwise, if $c=a$, the type $(a-1, b, c)$ is an upper type.
\end{itemize}
Hence, we will need to consider the cases when one of the above conditions doesn't hold. For $(a, b, c)$ a proper lower type, this can happen in the following cases: $c=0$, $a=0$, $b=0$ or $a=c$. These are the special cases considered separately in the proof of Lemma \ref{wc:lem:ternary-auxiliary-positive}, where $U_n$ is also (more directly) shown to be positive.

\begin{lemma}
\label{wc:lem:ternary-auxiliary-positive}
    For $n \geqslant k$ and $(a, n-a-c, c)$ a proper lower type (i.e. with $0 \leqslant c \leqslant a$, $0 \leqslant n-a-c$ and $0 < a+c$) we have $U_n(a, c)>0$.
\leanblueprint{ternary-auxiliary-positive}
\end{lemma}

\begin{remark}
    We note that since for $a=c=0$ (this follows from the case 2.2 of the proof below), $a<0$, $c<0$ or $n-a-c<0$ we have $U_n(a, c)=0$, it follows that $U_n(a, c) \geqslant 0$ whenever $a \geqslant c$.
\end{remark}

\begin{proof}
    Similarly to the $d=1$ case, we will induct on $n$, starting from $n=k$.

    \begin{caseindent}
    $1:$ $n=k$. In this case, we have $U_n(a+k+1, c-k-1) = U_n(a+k+1, c-k)=0$ and so
        \begin{equation*}
            U_n(a, c)=\binom{n}{a, c}-\binom{n}{a+1, c-1} - U_n(a+1, c)
        \end{equation*}
        for any $a, c$. Hence:
        \begin{multline*}
            U_n(a, c) = \binom{n}{a, c}-\binom{n}{a+1, c-1} - U_n(a+1, c)\\
            = \binom{n}{a, c}-\binom{n}{a+1, c-1} - \binom{n}{a+1, c} + \binom{n}{a+2, c-1} + U_n(a+2, c)\\
            = \sum_{i \geqslant a} \Big( \binom{n}{i, c} - \binom{n}{i+1, c-1} \Big) (-1)^{i-a}\\
            = \sum_{i \geqslant a} \binom{n}{i, c} (-1)^{i-a} - \sum_{i \geqslant a+1} \binom{n}{i, c-1} (-1)^{i-a-1} \\
            = \binom{n}{c}\binom{n-c-1}{a-1} - \binom{n}{c-1}\binom{n-c}{a},
        \end{multline*}
        where the last equality can easily be shown by induction on $a$. The last difference equals
        \begin{equation*}
            \frac{n! \cdot ((a-c)(n-c) + a)}{c!a!(n-a-c)! \cdot (n-c)(n-c+1)},
        \end{equation*}
        which is positive for $c \leqslant a \leqslant n-c$.
    
    $2:$ $n>k$. To do induction on $n$, let us recall Equation \ref{d2_u_n1}.
    \begin{caseindent}
        $2.1:$ If we have: $0 \leqslant a-1, c-1, n-a-c-1$ and $c \leqslant a-1$, then $(a-1, n-a-c, c)$, $(a, n-a-c-1, c)$ and $(a, n-a-c, c-1)$ are all lower types in $\{0,1,2\}^{n-1}$ and so by induction each of the three summands in \ref{d2_u_n1} is positive, and so is $U_n(a, c)$. Let us now consider the `edge' cases when one of these conditions fail.
    
        $2.2:$ If $c=0$, then
        \begin{multline*}
            U_n(a, 0) = \binom{n}{a} - 0 - U_n(a+1, 0)=\binom{n}{a}-\binom{n}{a+1}+U_n(a+2, 0)\\
            = \binom{n}{a}-\binom{n}{a+1} + \binom{n}{a+2} - \binom{n}{a+3} + \ldots + \binom{n}{n}(-1)^{n-a}=\binom{n-1}{a-1}.
        \end{multline*}
        This is positive unless $a=c=0$ when it equals $0$ (which is the special case allowed in Proposition \ref{the_weights_assigning_proposition}).
        
        $2.3:$ If $a=0$ then (since $0 \leqslant c \leqslant a$) we must also have $c=0$ which we have considered in the previous case.
    
        $2.4:$ $n-a-c=0$. We may assume that $a, c > 0$. But then $\un(a, c) \geqslant \uno(a, c-1)>0$ by induction, which finishes the proof.
    
        $2.5:$ $c=a$. We may assume that $a>0$ (as $a=c=0$ has been discussed in the case $2.2$).
        We have
            \begin{multline*}
                \uno(a-1, a)=\binom{n}{a-1, a} - \binom{n}{a, a-1}\\ - \uno(a, a) + \uno(a+k+1, a-k) + \uno(a+k+1, a-k-1).
            \end{multline*}
            Hence, by \ref{d2_u_n1},
            \begin{multline*}
                \un(a, a) = \uno(a-1, a) + \uno(a, a) + \uno(a, a-1)\\
                = \uno(a+k+1, a-k) + \uno(a+k+1, a-k-1) + \uno(a, a-1).
            \end{multline*}
            Since $a+k+1 \geqslant a-k$ and $a+k+1 \geqslant a-k-1$, we have $\un(a, a) \geqslant \uno(a, a-1)$. Since also $(n-1)-a-(a-1)=n-2a=n-a-c \geqslant 0$, $a > 0$ and $a>a-1$ we have $\uno(a, a-1) \geqslant 0$ by induction. \qedhere
    \end{caseindent}
\end{caseindent}
\end{proof}

\subsection{Showing that chains starting at the inner types have positive weights}

As in the case when $d=1$, to find the weights of the chains starting at the inner layers we need to take into account the fact that some chains starting on the other side of the cube pass through these layers.

\begin{lemma}
\label{wc:lem:ternary-inner-weight}
    For $(a, n-a-c, c)$ a lower inner type (i.e. with $a-k < c \leqslant a$) we have
    \begin{equation*}
        W_n(a, c)=U_n(a, c) - U_n(c+k, a-k).
    \end{equation*}
\leanblueprint{ternary-inner-weight}
\end{lemma}

\begin{proof}
    We will prove this by induction on the layer (i.e. $n-a+c$), starting from $a=c+k-1$ (that is $n-a+c=n-k+1$) and going towards $a=c$ (that is $n-a+c=n$).

    For $c=a-k+1$, since all the types in Equation \ref{the_key_recursion_lowest_inner} except for $(a, b, c)$ are outer (and for lower outer types $U_n$ and $W_n$ agree), we have:
    \begin{multline*}
        \wn(a, c)\\
        = \binom{n}{a, c}-\binom{n}{a+1, c-1} - \un(a+1, c) + \un(a+k+1, c-k) + \un(a+k+1, c-k-1) - \un(a-k, c + k)\\
        = \un(a, c) - \un(c+k, a-k),
    \end{multline*}
    as required.

    Now let us consider a type $(a, b, c)$ with $a-k+1 < c \leqslant a$. By induction we have
    \begin{equation*}
        \wn(a+1, c) = \un(a+1, c) - \un(c+k, a-k+1).
    \end{equation*}
    Let us recall that $W_n$ and $U_n$ agree on lower outer types (and $(a+k+1, c-k-1),$ $(a+k+1, c-k)$ are such types), and that $W_n(a, c)=U_n(c, a)$ for upper outer types (and $(a-k, c+k)$, $(a-k+1, c+k)$ are such types). Hence, by Equation \ref{the_key_recursion_inner}:
    \begin{multline*}
        \wn(a, c)\\ = \binom{n}{a, c}-\binom{n}{a+1, c-1} + \un(a+k+1, c-k) + \un(a+k+1, c-k-1)\\ - \big( \un(a+1, c) - \un(c+k, a-k+1) \big) - \un(c+k, a-k) - \un(c+k, a-k+1)\\
        = \un(a, c) - \un(c+k, a-k),
    \end{multline*}
    as required.
\end{proof}

For the improper type with $a=c=0$, we have
\begin{equation*}
    W_n(0, 0)=U_n(0, 0)-U_n(k, -k)=U_n(0, 0)=0.
\end{equation*}
Now we need to show that for any proper lower inner type $(a, n-a-c, c)$ we have $U_n(a, c) - U_n(c+k, a-k) > 0$. The spirit of the proof is very similar to that for the outer types.

Let us recall that a type $(a, n-a-c, c)$ is proper lower inner if
\begin{equation}
\label{proper_lower_inner_conditions}
    0 \leqslant a, c \leqslant n, \ 0<a+c\leqslant n \ \text{and} \ a-k<c\leqslant a.
\end{equation}
Let us also recall Equation \ref{d2_u_n1}:
\begin{equation*}
    U_n(a, c)=\uno(a-1, c) + \uno(a, c)+\uno(a, c-1).
\end{equation*}
Similarly, we have:
\begin{equation*}
    \un(c+k, a-k)=\uno(c+k, a-k-1)+\uno(c+k, a-k)+\uno(c+k-1, a-k).
\end{equation*}
Putting these together, we have:
\begin{equation}
\label{d2_w_recursive}
    W_n(a, c)=\wno(a-1, c)+\wno(a, c)+\wno(a, c-1).
\end{equation}
We will fix $k$ and induct on $n$ (with $n=k$ the base case). If $(a, b, c)$ (with $b=n-a-c$) is a proper lower inner type such that the types $(a-1, b, c)$, $(a, b-1, c)$ and $(a, b, c-1)$ are also proper lower inner, then by induction and Equation \ref{d2_w_recursive} $\wn(a, c)$ is positive. This happens if all the following conditions are met:
\begin{itemize}
    \item $0 \leqslant a-1, b-1, c-1$. Otherwise, one of the types $(a-1, b, c)$, $(a, b-1, c)$ and $(a, b, c-1)$ is an empty type.
    \item $1 < a+c$. Otherwise, $(a-1, n-a-c, c)$ or $(a, n-a-c, c-1)$ is the improper type $(0, n-1, 0)$.
    \item $(a-1)-k < c \leqslant a-1$. Otherwise, $(a-1, b, c)$ is an outer type. We note that for $(a, b, c)$ a lower inner type the first inequality always holds.
    \item $a-k < c-1 \leqslant a$. Otherwise, $(a, b, c-1)$ is an outer type. We note that for $(a, b, c)$ a lower type the second inequality always holds.
\end{itemize}
We note that we also require $a, b, c < n$, but this is ensured by the first condition.

We notice that for a proper lower inner type $(a, b, c)$ (i.e. satisfying all the inequalities from Equation \ref{proper_lower_inner_conditions}) and $n>k$, the above conditions may not hold only in one of the following five cases: $a=c$, $a=c+k-1$, $b=0$ or $c=0$. As we will see in the proof of the lemma, showing that $\wn(a, c)$ is positive in these cases is pretty straightforward (it only requires care with handling the cases).

\begin{lemma}
\label{wc:lem:ternary-inner-weight-positive}
    For $1 < k \leqslant n$ and $(a, n-a-c, c)$ a proper lower inner type, i.e. $0 \leqslant a, c \leqslant n$, $0<a+c\leqslant n$ and $a-k<c\leqslant a$ we have $\wn(a, c)>0$.
\leanblueprint{ternary-inner-weight-positive}
\end{lemma}

\begin{proof}
    We prove the lemma by induction on $n$, with $n=k$ the base case.

    $1:$ If $n=k$, since for a proper lower inner type $a-n<c\leqslant a$ we cannot have $a=n$ (as then $0<c$, but we also need $a+c \leqslant n$). Hence, we have $\un(c+n, a-n)=0$ and so $W_n(a, c)=U_n(a, c)>0$.

    $2:$ Let us now suppose that $n>k$.
        \begin{caseindent}
        $2.1:$ If all the types $(a-1, n-a-c, c),$ $(a, n-a-c-1, c)$ and $(a, n-a-c, c-1)$ are proper lower inner in $\{0, 1, 2\}^{n-1}$, then by induction and Equation \ref{d2_w_recursive} $\wn(a, c)$ is positive.

        $2.2:$ We recall that if one of $(a-1, n-a-c, c),$ $(a, n-a-c-1, c)$ and $(a, n-a-c, c-1)$ is not a proper lower inner type, then we must have: $a=c$, $a=c+k-1$, $b=0$ or $c=0$. We now consider these case by case.
        \begin{caseindent}
        $2.2.1:$ $a=c$. From the recursive definition of $U_n$ we have:
            \begin{multline*}
                U_{n-1}(a-1, a)=\binom{n-1}{a-1, a}-\binom{n-1}{a, a-1}-\uno(a, a)\\ + \uno(a+k, a-k-1)+\uno(a+k, a-k).
            \end{multline*}
            Hence
            \begin{multline*}
                \un(a, a)
                =\uno(a-1, a)+\uno(a, a)+U_{n-1}(a, a-1)\\
                =\uno(a+k, a-k-1)+\uno(a+k, a-k) + \uno(a, a-1),
            \end{multline*}
            and so
            \begin{multline*}
                \wn(a, a)=\un(a, a) - \un(a+k, a-k)\\
                = \uno(a+k, a-k-1)+\uno(a+k, a-k)+\uno(a, a-1)\\
                - \uno(a+k-1, a-k) - \uno(a+k, a-k) - \uno(a+k, a-k-1)\\
                = \uno(a, a-1) - \uno(a+k-1, a-k)=\wno(a, a-1).
            \end{multline*}
            From the induction, this is positive unless $c=a=0$ (when it equals $0$).
    
        $2.2.2:$ $a=c+k-1$. In this case we have
            \begin{equation*}
                \uno(a, c-1)=\uno(c+k-1, a-k)
            \end{equation*}
            and so
            \begin{multline*}
                \un(a, c)-\un(c+k, a-k)\\
                = \uno(a-1, c) + \uno(a, c) + \uno(a, c-1)\\
                - \uno(c+k-1, a-k) - \uno(c+k, a-k) - \uno(c+k, a-k-1)\\
                =\uno(a-1, c) - \uno(c+k, (a-1)-k) + \uno(a, c) - \uno(c+k, a-k)\\
                =\wno(a-1, c)+\wno(a, c).
            \end{multline*}
            From the inductive hypothesis, the only case when these two could both be non-positive is when $a=0$ and $a+c=n$. However, this would imply $c=n>a$, which cannot happen.

        $2.2.3:$ $n-a-c=0$. We may assume that $a>c$ (as we have considered $a=c$ in $2.2.1$) and that $a>0$ (as we have considered $a=0$ in $2.2.2$). We clearly have $\wno(a, c)=0$ (as this is an empty type). Moreover, $\wno(a, c-1) \geqslant 0$ (as $a \geqslant c-1$). Finally, we notice that $0 \leqslant a-1, c \leqslant n-1$ (as $a>c\geqslant 0$), $0<(a-1)+c \leqslant n-1$ (as $a+c=n$) and $(a-1)-k<c \leqslant a-1$, so $(a-1, n-a-c, c)$ is a proper lower inner type in $\{0, 1, 2\}^{n-1}$. Hence, $\wn(a, c) \geqslant \wno(a-1, c) > 0$.

        $2.2.4:$ $c=0$. Since $(a, n-a-c, c)$ is proper lower inner, $a>0$. If $a>1$, similarly to above we have $\wno(a, c-1)=0$, $\wno(a, c) \geqslant 0$ and $\wno(a-1, c)>0$ (as we may easily check that $(a-1, n-a-c, c)$ is a proper lower inner type in $\{0, 1, 2\}^{n-1}$.

        Let us now suppose that $a=1$. However, then the type we consider is $(1, n-1, 0)$ and we have:
            \begin{equation*}
                \wn(1, 0) = \wno(0, 0) + \wno(1, 0) + \wno(1, -1).
            \end{equation*}
        We have $\wno(0,0)=0$ (as this is the improper type $(0, n-1,0)$) and $\wno(1, -1)=0$ (as this is an empty type). However, for $k \geqslant 2$ the type $(1, n-2, 0)$ is a proper lower inner type, and so $\wn(1, 0)=\wno(1, 0)>0$. \qedhere
        \end{caseindent}
    \end{caseindent}
\end{proof}

Hence, from Lemmas \ref{wc:lem:ternary-inner-weight} and \ref{wc:lem:ternary-inner-weight-positive} all the weights assigned to the basic chains in the course of the recursive procedure are positive (except for the weight of $\{\theonepoint\}$, which is $0$). As shown in Lemma \ref{positive_basic_enough_lemma}, this ensures that the assignment of weights satisfies conditions (\ref{non_negative_condition})-(\ref{positive_A_condition}) of Proposition \ref{the_weights_assigning_proposition}. Since the procedure also ensures that the sum of weights of chains passing through any point on the hypercube is $1$, all the conditions of Proposition \ref{the_weights_assigning_proposition} are satisfied and so, by Lemma \ref{weights_imply_theorem_lemma}, this proves Theorem \ref{the_main_theorem} for $d=2$.

\section{Conclusions and further research}
\label{section_conclusion}

In this paper, we have shown that the unique largest subset $A \subseteq \{0, 1, 2\}^n$ such that there are no distinct elements $x, y \in A$ with $x \leqslant y$ and $x_i < y_i$ for at most $k$ coordinates $i$ is
\begin{equation}
    A = \Big\{ x \in \{0,1, 2\}^n : \vert x\vert  \equiv n \ (\modd 2k+1) \Big\}.
\end{equation}
The main contribution of this paper is the decomposition of the cube $\cube$ into weighted chains of width at most $k$ that are either of maximal length $2k+1$ or symmetric. This is a novel approach compared to those used for other generalisations of Sperner's theorem. 

Weighted chain decomposition is preserved under permuting coordinates, allowing us to consider only `types' of points instead of points themselves.

Moreover, the decomposition is explicit and allows us to shift focus from the construction of such a decomposition to proving its properties (i.e., positive weights) with binomial coefficients.

Solving the case $d=1$ with weighted chains, we discovered a new proof of Sperner's theorem.

\subsection{Further research}
The most natural extension of this work would be to prove Theorem \ref{the_main_theorem} for $d>2$. Proposition \ref{the_weights_assigning_proposition} shows a way to do that: we only need to find an appropriate weight assignment. The proofs for $d=1$ and $d=2$ suggest that a recursive weight assignment, invariant under index permutation, should both be possible and relatively easy to work with (that is, showing that all its weights are positive should be a straightforward induction, provided that the weight assignment follows a simple recursive procedure).

However, the main drawback to simply generalising the construction with basic chains to larger $d$s is that, for $d>2$, $n$ a multiple of $kd+1$ and $x=(\lfloor \frac{d}{2} \rfloor -1, \dots, \lfloor \frac{d}{2} \rfloor-1) \in A_1$ there are no basic good chains starting at $x$. Hence, considering a different (or broader) family of chains might be necessary.

A potentially easier intermediate step would be to prove Theorem \ref{the_main_theorem} without the uniqueness, that is simply to construct an assignment of non-negative weights to good chains so that the sum of values of all chains passing through any vertex is $1$.

It is easy to show by induction on $n$ (with the base cases $k \leqslant n \leqslant 2k - 1$ and inductive step of size $k$) that Theorem \ref{the_main_theorem} holds asymptotically, i.e. that for any $d, k$ constant and $n \rightarrow \infty$ we have (denoting by $A_{k, d}(n)$ the size of the largest subsets satisfying the conditions of Theorem \ref{the_main_theorem} for the given $k, d, n$):
\begin{equation*}
    \frac{A_{k, d}(n)}{(1+d)^n} = \frac{1}{dk+1} + o(1).
\end{equation*}
\leanblueprint{asymptotic-density}
An interesting additional question to ask (for $d=2$) would be what makes the family of basic good chains sufficient for an assignment of non-negative weights to them being possible? For example, a similar family of \emph{anti-basic chains}, i.e. chains that, if starting at type $(a, b, c)$, go through types
\begin{multline*}
    (a, b, c), \ (a-1, b+1, c), \ (a-2, b+2, c), \ \ldots, \\
    (a-k, b+k, c), \ (a-k, b+k-1, c+1), \ (a-k, b+k-2, c+2), \ \ldots, \ (a-k, b, c+k)
\end{multline*}
is not sufficient, and assigning weights to these chains recursively as in Subsection \ref{d2_weights_assignement_subsection} would lead to some negative weights being assigned (which we checked in a Python simulation).

Another interesting line of investigation would be to try different methods for constructing bounded-width, unweighted chain decompositions of the hypercube, or to show that none exists.

\appendix

\section{A new proof of Sperner's theorem}
\label{appendix_sperner}

Below we present a new proof of Sperner's theorem that we arrived at while working on this paper. It is the special case with $d=1$ and $k=n$. This section is self-contained and doesn't refer to the rest of the paper, allowing for it to be read independently.

In what follows, we will use letters $X, Y, \ldots$ to refer to subsets of $[n]$ and $A, B, C, \ldots$ to refer to families of subsets. $\mathcal{A}, \mathcal{B}, \mathcal{C}, \ldots$ will denote collections of such families.

Let us call a family of subsets of $[n]$ an \emph{antichain} if it contains no two elements $X \neq Y$ with $X \subseteq Y$.

Hence, we will show that for any $n$, the unique largest antichains in $\mathcal{P}([n])$ are
\begin{equation*}
    A_1 := \big\{X \subseteq [n]: \vert X\vert  = \left\lfloor\frac{n}{2}\right\rfloor\big\}
    \ \text{and} \
    A_2 := \big\{X \subseteq [n]: \vert X\vert  = \left\lceil\frac{n}{2}\right\rceil\big\}
\end{equation*}
(where if $2 \mid n$, we have $A_1=A_2$).
\leanblueprint{sperner-appendix}

Let us call a family of subsets $C = \{X_1, \ldots, X_l\}$ of $[n]$ a \emph{symmetric chain} if:
\begin{itemize}
    \item $X_1 \subseteq \ldots \subseteq X_l$.
    \item $\vert X_{i+1}\vert =\vert X_i\vert +1$ for any $1 \leqslant i \leqslant l-1$.
    \item $\vert X_1\vert +\vert X_l\vert =n$.
\end{itemize}
These conditions imply $\vert X_i\vert +\vert X_{l+1-i}\vert =n$ for any $1 \leqslant i \leqslant l$.

Of course, for any symmetric chain $C$ and any antichain $A$ we must have $\vert A \cap C \vert \leqslant 1$.

Let $\mathcal{C}$ denote the collection of all symmetric chains, and for $0 \leqslant i \leqslant \frac{n}{2}$ let $\mathcal{C}_i$ denote the collection of all symmetric chains whose smallest element is of size $i$. Hence, if $C \in \mathcal{C}_i$, then $C = (X_1, \ldots, X_{n-2i+1})$ where $\vert X_j\vert =i+j-1$ and $X_j \subseteq X_{j+1}$ for any $j$.

For $2 \mid n$ we have
\begin{equation}
\label{n_even_cn2_equation}
    \mathcal{C}_{\frac{n}{2}} = \big\{\{X\}: \vert X\vert =\frac{n}{2}\big\}
\end{equation}
and for $2 \nmid n$ we have
\begin{equation}
\label{n_odd_cn2_equation}
    \mathcal{C}_{\lfloor n/2 \rfloor} = \big\{ \{X, Y\} : \ \vert X\vert  = \frac{n-1}{2}, \ \vert Y\vert =\frac{n+1}{2}, \ X \subseteq Y \big\}.
\end{equation}

Let us notice that for any $i$, $\{X \subseteq [n]: \vert X\vert =i\}$ is an orbit under the group action of permuting $[n]$. Similarly, so is $\mathcal{C}_i$ for any $i$. It follows that for any $X, Y \subseteq [n]$ with $\vert X\vert =\vert Y\vert $ and any $i$ exactly the same number of symmetric chains from $\mathcal{C}_i$ contain $X$ and $Y$.

Now let us assign weights to the symmetric chains. Let $w: \mathcal{C} \rightarrow (0, 1]$ be a weight function such that for $C \in \mathcal{C}_i$
\begin{equation}
    w(C) = \frac{1}{\vert \mathcal{C}_i\vert } \Big(\binom{n}{i} - \binom{n}{i-1}\Big)
\end{equation}
(where $\binom{n}{-1}=0$).

Since $0 \leqslant i \leqslant \frac{n}{2}$, $w(C)$ is positive for any $C \in \mathcal{C}$.

\begin{claim}
    The total weight of symmetric chains passing through any subset of $[n]$ is $1$.
\end{claim}

\begin{proof}
    Let $L_i := \{X \subseteq [n]: \vert X\vert =i\}$ be called \emph{layer $i$}. For $0 \leqslant i \leqslant \frac{n}{2}$, the symmetric chains containing an element of $L_i$ are exactly those from $\mathcal{C}_0 \cup \ldots \cup \mathcal{C}_i$. Hence, the sum of weights of chains intersecting $L_i$ is
    \begin{equation*}
        \sum_{j=0}^i \vert \mathcal{C}_j\vert  \cdot \frac{1}{\vert \mathcal{C}_j\vert } \cdot \Big( \binom{n}{j}-\binom{n}{j-1} \Big)=\binom{n}{i}.
    \end{equation*}
    From the symmetry of the chains under the permutation of $[n]$, the total weight of symmetric chains passing through any element of $L_i$ is the same and equals $\frac{\binom{n}{i}}{\vert L_i\vert }=\frac{\binom{n}{i}}{\binom{n}{i}}=1$.
\end{proof}

Let $A \subseteq \mathcal{P}([n])$ be an antichain of maximal size. We have
\begin{equation}
    \vert A\vert  = \sum_{X \in A} 1 = \sum_{X \in A} \sum_{C \in \mathcal{C}: X \in C} w(C) = \sum_{C \in \mathcal{C}} w(C) \sum_{X \in A} 1_{X \in C} = \sum_{C \in \mathcal{C}} w(C) \cdot \vert C \cap A\vert  \leqslant \sum_{C \in \mathcal{C}} w(C),
\end{equation}
where equality holds iff $A$ contains an element of every chain from $\mathcal{C}$. Similarly, since $A_1$ and $A_2$ satisfy this condition, we get
\begin{equation}
    \vert A_1\vert =\sum_{C \in \mathcal{C}} w(C) = \vert A_2\vert ,
\end{equation}
and so $\vert A\vert  \leqslant \vert A_1\vert  = \vert A_2\vert $, with an equality only if $A$ contains an element of each symmetric chain.

For $2 \mid n$, it immediately follows from \ref{n_even_cn2_equation} that $\{X \subseteq [n]: \vert X\vert  = \frac{n}{2}\} \subseteq A$ and so $A=A_1=A_2$.

For $2 \nmid n$, it follows from \ref{n_odd_cn2_equation} that for any $X \subseteq Y$, $\vert X\vert =\frac{n-1}{2}$, $\vert Y\vert =\frac{n+1}{2}$, $A$ must contain either $X$ or $Y$. Considering the bipartite graph of subsets of $[n]$ of size $\frac{n-1}{2}$ and $\frac{n+1}{2}$, we may easily see that we either have $\{X \subseteq [n]: \vert X\vert  = \frac{n-1}{2}\} \subseteq A$ (and so $A=A_1$) or $\{X \subseteq [n]: \vert X\vert  = \frac{n+1}{2}\} \subseteq A$ (and so $A=A_2$).

This finishes the proof.

We note that an interesting difference between this approach and the standard chain decomposition proof is that this one is symmetric, shifting the focus from constructing a decomposition to proving its properties.

\section{A proof of Theorem \ref{the_main_theorem} for any \texorpdfstring{$d$}{d} and \texorpdfstring{$\frac{n}{2} \leqslant k$}{n/2 <= k}}
\label{appendix_large_k}

Below we present a simple proof that the sets $A_1, A_2$ from Theorem \ref{the_main_theorem} are optimal (but which doesn't show their uniqueness) for any $d$ and $\frac{n}{2} \leqslant k$ with an alternative method. This result also follows from \cite[Theorem 5.1]{tour_m_part_p_sperner_families}. We provide a proof at which we arrived independently. We note that the same proof would also work for cuboids $\{0, 1, \ldots, d_1\} \times \ldots \times \{0, 1, \ldots, d_n\}$.
\leanblueprint{large-k-appendix}

The proof uses the decomposition of the hypercube into symmetric chains $S_1, S_2, \ldots, S_m$, i.e. sets $S_i=(x_i^1, \ldots, x_i^{l_i})$ such that $x_i^1 \leqslant \ldots \leqslant x_i^{l_i}$ and $\vert x_i^j\vert +\vert x_i^{l_i+1-j}\vert =dn$ for any $1 \leqslant j \leqslant l_i$. Its existence was shown by de Bruijn, Tengbergen and Kruyswijk in \cite{symmetric_chain_decomposition}.

Let $A'$ be any largest possible $k$-separated subset of the hypercube. Let $X=\{0, \ldots, d\}^{\lfloor n/2 \rfloor}, Y=\{0, \ldots, d\}^{\lceil n/2 \rceil}$ and let us consider the decomposition $\{0, \ldots, d\}^n=X \times Y$. For any two tuples $(x, y), (x', y') \in \{0, \ldots, d\}^n$ (where $x, x' \in X,$ $y, y' \in Y$), if $x=x'$ and $y \leqslant y'$ or $y=y'$ and $x \leqslant x'$, they cannot both belong to $A'$.

Let us decompose $X$ and $Y$ into symmetric chains:
\begin{equation*}
    X = C_1 \sqcup \ldots \sqcup C_p, \ Y = D_1 \sqcup \ldots \sqcup D_q.
\end{equation*}
We cannot have more than $\min(\vert C_i\vert , \vert D_j\vert )$ elements $(x, y) \in A'$ with $x \in C_i, y \in D_j$. On the other hand, $A_1$ has exactly this number of such elements. Hence, we have
\begin{equation}
    \vert A'\vert = \sum_{i, j} \vert A' \cap (C_i \times D_j)\vert  \leqslant \sum_{i, j} \min(\vert C_i\vert , \vert D_j\vert ) =  \sum_{i, j} \vert A_1 \cap (C_i \times D_j)\vert  = \vert A_1\vert ,
\end{equation}
so $A_1$ is a $k$-separated set of maximal size, and similarly for $A_2$.

\end{document}